\documentclass[11pt, reqno]{amsart}

\title{On the inverse Goldbach problem}


\usepackage[T1]{fontenc}
\usepackage{amsmath}
\usepackage{amssymb}
\usepackage{amsthm}
\usepackage[left=3.5cm, right=3.5cm, paperheight=11.8in]{geometry}
\usepackage{hyperref}
\usepackage{fancyhdr}
\usepackage{enumitem}
\usepackage{comment}
\usepackage{nicefrac}
\usepackage{mathrsfs}
\usepackage{bm}
\usepackage{graphicx}
\usepackage[utf8]{inputenc}
\usepackage{cancel}
\usepackage{mathtools}

\AtBeginDocument{%
   \def\MR#1{}
}

\newtheorem{thm}{Theorem}[section]
\newtheorem{cor}[thm]{Corollary}%[section]
\newtheorem{lem}[thm]{Lemma}

\theoremstyle{definition} 
\newtheorem{defi}[thm]{Definition}%[section]
\let\olddefi\defi
\renewcommand{\defi}{\olddefi\normalfont}
\newtheorem{example}[thm]{Example}
\let\oldexample\example
\renewcommand{\example}{\oldexample\normalfont}
\newtheorem{rmk}[thm]{Remark}
\let\oldrmk\rmk
\renewcommand{\rmk}{\oldrmk\normalfont}

\author[P.~Leonetti]{Paolo Leonetti}
\address{
Universit\'{a} degli Studi dell’Insubria\\ via Monte Generoso 71 \\ Varese 21100\\ Italy}
\email{leonetti.paolo@gmail.com}
\hypersetup{
    pdftitle={On the inverse Goldbach problem},
    pdfauthor={Paolo Leonetti},
    pdfmenubar=false,
    pdffitwindow=true,
    pdfstartview=FitH,
    colorlinks=true,
    linkcolor=blue,
    citecolor=green,
    urlcolor=cyan
}

\providecommand{\MR}[1]{}

\providecommand{\MR}{\relax\ifhmode\unskip\space\fi MR }

\begin{document}

\begin{abstract}
We establish two quantitative results concerning additive decompositions of the set $\mathbb{P}$ of primes. 
First, we prove that, if $A,B\subseteq\mathbb N$ have at least two
elements and $A+B$ differs from the primes by only finitely many elements,
then
$$
 \frac{\sqrt{x}}{\log x}\ll A(x),B(x)
 \ll \sqrt{x}.
$$

Second, we show that, if $k\ge 3$ and $A_1,\ldots,A_k\subseteq\mathbb N$ have at
least two elements each, then 
$$
 \bigl((A_1+\cdots+A_k)\mathbin{\bigtriangleup}\mathbb P\bigr)(x)
 \gg \left(\frac{x}{\log x}\right)^{1-1/k}.
$$
The proofs use Montgomery's large sieve, Gallagher's larger sieve, a mean-value argument for a certain multiplicative function, and a submultiplicative
sumset inequality.
\end{abstract}

\maketitle
\thispagestyle{empty}

\section{Introduction and Main Results}\label{sec:intro} 

A subset $S$ of the nonnegative integers $\mathbb{N}$ is said to be \emph{totally irreducible} if there are no sets $A,B\subseteq \mathbb{N}$ with cardinality at least $2$ such that 
$$
S=_\star A+B,
$$
meaning that the symmetric difference $S\bigtriangleup (A+B)$ is finite. 
Here, as usual, $A+B$ stands for the sumset $\{x+y: x \in A, y \in B\}$. 
It is known that “almost all” subsets of $\mathbb{N}$ are totally irreducible, from both the topological and measure theoretic viewpoints, see  \cite{MR4607615}. 
An old conjecture of Ostmann, sometimes referred to as the inverse Goldbach problem, states that the set $\mathbb{P}$ of primes is totally irreducible, see e.g. \cite{MR1996367, MR2409182, MR2522039, MR3378834, MR3248483, MR3553391} and references therein. Erd\"{o}s defines the problem as “an old and very fascinating conjecture,” see \cite{MR532690}; cf. also Problem 431 in \cite{Bloom431}.

Very recently, OpenAI announced a proof of Ostmann's conjecture in \cite{OpenAIOstmann2026}; at the time of writing, this claim has not yet 
been verified by the math community. 
%undergone the usual journal refereeing process. 
Independently of this recent development, it is remarkable that, 
%Although Ostmann's conjecture remains widely open, it is remarkable that, 
as it follows by a result of S{\'{a}}rk{\"{o}}zy \cite{MR146164}, there exists a totally irreducible set $S\subseteq \mathbb{N}$ such that 
$$
(S\bigtriangleup \mathbb{P})(x) \ll \frac{x}{\log x\sqrt{\log \log  x}}
\quad \text{ as }x\to \infty, 
$$
where $A(x):=|A\cap [0,x]|$ for each $A\subseteq \mathbb{N}$ and real $x\ge 0$. Similarly, Bienvenu proved in \cite[Theorem 4]{MR4689545}, in a precise sense, that “almost all small perturbations” of the primes are totally irreducible. 

Elsholtz and Harper proved in \cite[Theorem 2.6]{MR3378834} that, if $A,B\subseteq \mathbb{N}$ have cardinality at least $2$ and $A+B=_\star \mathbb{P}$ then 
\begin{equation}\label{eq:startingbound}
\frac{\sqrt{x}}{\log x \cdot\log\log x} \ll 
\, A(x), B(x)\, 
\ll \sqrt{x}\,\log\log x 
\quad \text{ as }x\to \infty.
\end{equation}
Here, the implied constants are absolute. 

In our first result, we improve the above upper and lower bounds.\footnote{On the other hand, Elsholtz and Harper methods apply also to the \textquotedblleft finitary\textquotedblright\,analogous decomposition, namely, $A+B=\mathbb{P}\cap [x_0,x]$ for sufficiently large (fixed) $x_0$, cf. \cite[Footnote 1]{MR3378834}.} 
Below, the implied constants depend on $A,B$, hence we replace $\ll$ by  $\ll_{A,B}$. 
\begin{thm}\label{thm:main2}
    Let $A,B\subseteq \mathbb{N}$ be sets with cardinality at least $2$ such that $A+B=_\star \mathbb{P}$. Then 
    \begin{equation}\label{eq:newsindeptartingbound}
\frac{\sqrt{x}}{\log x} 
\,\ll_{A,B}\, A(x), B(x)\, \ll_{A,B}\, 
\sqrt{x}
\quad \text{ as }x\to \infty.
\end{equation}
\end{thm}
As an immediate consequence, we obtain the following constant-free
formulation.
\begin{cor}\label{cor:kappa}
    Let $A,B\subseteq \mathbb{N}$ be sets with cardinality at least $2$ such that $A+B=_\star \mathbb{P}$. Pick also a function $\kappa: [1,\infty)\to [1,\infty)$ with $\lim_x \kappa(x)=\infty$. 
    Then 
    $$
\frac{\sqrt{x}}{\log x\cdot \kappa(x)} 
\,\le\, A(x), B(x)\, \le\, 
\sqrt{x}\cdot \kappa(x)
\quad \text{ as }x\to \infty.
$$
\end{cor}

In a related direction, Elsholtz proved in \cite{MR1996367} that there are no sets $A,B,C\subseteq \mathbb{N}$ with cardinality at least $2$ such that $A+B+C=_\star \mathbb{P}$, hence solving the “ternary inverse Goldbach problem.” 
Then, Shao showed that there exists a
sufficiently small $c>0$ such that, if
$A,B,C\subseteq\mathbb N$ have cardinality at least $2$ and satisfy 
$
A+B+C\subseteq\mathbb P,
$ 
then
$$
\min\{A(x),B(x),C(x)\}
\ll
x^{1/3-c}
\quad\text{as }x\to\infty,
$$
see \cite[Theorem 1.3]{MR3553391}; cf. also
\cite[Theorem 3]{MR941928}: in fact, taking into account \eqref{eq:startingbound}, this gives another proof of the ternary inverse Goldbach problem, namely, 
\begin{equation}\label{eq:weakternary}
((A+B+C)\bigtriangleup \mathbb{P})(x)\to \infty 
\quad\text{as }x\to\infty
\end{equation}
for all $A,B,C\subseteq \mathbb{N}$ with cardinality at least $2$.

Intuitively, the strong additive structure of the sumset $A+B+C$ should be “very far” from $\mathbb{P}$, which should behave as a “random set” from an additive viewpoint. In our last main result, we give a quantitative strengthening of \eqref{eq:weakternary}. 
\begin{thm}\label{thm:main3}
    Fix $k\ge 3$ and let $A_1,\ldots,A_k\subseteq\mathbb N$ have at least
two elements each. Then 
    \begin{equation}\label{eq:finalclaim}
    ((A_1+\cdots+A_k)\bigtriangleup \mathbb{P})(x) \gg 
    \left(\frac{x}{\log x}\right)^{1-1/k}
    %\frac{\sqrt{x}}{(\log x)^2}
    \quad\text{as }x\to\infty.
    \end{equation}
    %In particular, 
    %$$
    %((A+B+C)\bigtriangleup \mathbb{P})(x) \gg \left(\frac{x}{\log x}\right)^{2/3}
    %\quad\text{as }x\to\infty
    %$$
    %for all $A,B,C\subseteq \mathbb{N}$ with cardinality at least $2$.
\end{thm}

In a related direction, Croot and Elsholtz proved in \cite{MR2129526} that, assuming the Hardy--Littlewood prime-tuples conjecture, there exist $A,B,C\subseteq\mathbb N$, each having at least two elements, such that $A+B+C\subseteq\mathbb P$ and $(A+B+C)(x)\gg x/(\log x)^3$ as $x\to\infty$. %In a related direction, Croot and Elsholtz proved in \cite{MR2129526} that, assuming that Hardy--Littlewood conjecture, there exists $A,B,C\subseteq \mathbb{N}$, each having at least two elements, such that $A+B+C\subseteq \mathbb{P}$ and $(A+B+C)(x) \gg x/ (\log x)^3$ as $x\to \infty$. 

As immediate consequences of Theorem \ref{thm:main3}, we obtain that sets which do not \textquotedblleft differ much\textquotedblright\, from $\mathbb{P}$ cannot be written as $A_1+\cdots+A_k$. For instance, we get the following.
\begin{cor}\label{cor:powers}
Fix $k\geq3$ and let $A_1,\ldots,A_k\subseteq\mathbb N$ have at least two elements each. Let $P^\sharp$ be the set of prime powers. Then
$$
\bigl((A_1+\cdots+A_k)\bigtriangleup P^\sharp\bigr)(x)
\gg_k \left(\frac{x}{\log x}\right)^{1-1/k}
\quad\text{as }x\to\infty.
$$
In particular, there are no sets $A,B,C\subseteq\mathbb N$ with at least
two elements each such that $$P^\sharp=_\star A+B+C.$$
\end{cor}

The proofs of our results are given in the next sections. 

\section{Proof of Theorem \ref{thm:main2} and Corollary \ref{cor:kappa}}\label{sec:proof-main2}

We begin with two auxiliary lemmas. In the first one, we modify quantitatively the argument of Elsholtz and Harper in the proof of \cite[Theorem 2.6]{MR3378834}. 
As usual, we write 
$ 
u_+:=\max\{u,0\}
$ 
for each $u\in\mathbb R$.

\begin{lem}\label{lem:tail-defect}
Let $A,B\subseteq\mathbb N$ be infinite sets such that 
$$
A+B\subseteq\mathbb P 
\qquad \text{ and }\qquad 
A(x),B(x)\gg
\frac{\sqrt{x}}{(\log x)^2}\,\,\text{ as }x\to \infty. 
$$

For every prime $p$, define 
$$
d_{A,p}:=
\left(\frac12-\frac{|A\bmod p\,|}{p}\right)_+
\quad\text{ and }\quad 
d_{B,p}:=
\left(\frac12-\frac{|B\bmod p\,|}{p}\right)_+.
$$
Then
$$
\sum_{p>Z}
\frac{d_{A,p}+d_{B,p}}p
\ll
\sqrt{\frac{\log\log Z}{\log Z}} \quad \text{ as }Z\to \infty. 
$$
\end{lem}

\begin{proof}
For each prime $p$, set $A_p:=A\bmod{p}$, $B_p:=B\bmod{p}$, and $R_p:=A_p\cap (-B_p)$. Accordingly, define also 
$$
a_p:=|A_p|, \quad 
b_p:=|B_p|, \quad \text{ and }\quad 
r_p:=|R_p|.
$$
It follows that 
$$
a_p+b_p=|A_p|+|-B_p|=|A_p\cup (-B_p)|+|A_p\cap (-B_p)|\le p+r_p.
$$
In addition, if a residue class belongs to $R_p$, then
there exist $a\in A$ and $b\in B$ such that 
$
p\mid a+b.
$ 
Since
$
a+b
$ is prime, 
necessarily
$
p=a+b.
$ 
Moreover, the corresponding residue classes contain only one possible
element participating in such a representation. Indeed, if
$a'\in A$ and
$
a'\equiv a\bmod p,
$ 
then
$
p\mid a'+b, 
$ and with the same reasoning we get $a=a'$. 
The analogous assertion holds for $B$.
Hence $r_p\leq\min\{A(p),B(p)\}$. 
Now, fix $b_0\in B$. Since 
$
a+b_0\in\mathbb{P}$ for all $a \in A$, 
the map $
a\mapsto a+b_0
$ 
is injective. Thus, for all sufficiently large $p$, 
$
A(p)
\leq
\pi(p+b_0)
\ll p/\log p, 
$ 
and similarly 
$
B(p)\ll p/\log p.
$ 
It follows that
\begin{equation}\label{eq:rp-small}
r_p\ll\frac{p}{\log p}
\end{equation}
for all sufficiently large primes $p$.

Fix a large parameter $Q$ and put
$$
N:=
\left\lceil
Q^2(\log Q)^8
\right\rceil,
$$
so that, by hypothesis, 
$
A(N),B(N)\gg Q(\log Q)^2.
$ 
Now, apply Gallagher's larger-sieve calculation used in
\cite[Section 7]{MR3378834} (see \cite[Lemma 3.3]{MR3378834} for the statement)
to the finite sets 
$
A\cap[0,N]$ and $
B\cap[0,N]$, with the primes $p\leq Q$ as sieving primes. In such case, we obtain 
$$
\sum_{p\leq Q}\frac{\log p}{a_p}
\leq
\sum_{p\leq Q}\frac{\log p}{(|A\cap [0,N]) \bmod{p}\,|}
\leq
\left(1-\frac1{A(N)}\right)\log N
+\frac{\vartheta(Q)}{A(N)},
$$
where $\vartheta(Q):=\sum_{p\le Q}\log p$. Note that $\vartheta(Q)\ll Q$. Summing it with the analogous estimate for $B\cap [0,N]$ and ignoring $-\log(N)/A(N)$, it follows that 
\begin{equation}\label{eq:larger-sieve-A}
    \begin{split}
        \sum_{p\leq Q}
\left(
\frac1{a_p}+\frac1{b_p}
\right)\log p
&\leq
2\log N
+\vartheta(Q)
\left(
\frac1{A(N)}+\frac1{B(N)}
\right)\\
&\le 4\log Q+16\log\log Q+O\left(\frac{Q}{Q(\log Q)^{2}}\right)\\
&\le 4\log Q+O(\log\log Q).
    \end{split}
\end{equation}

At this point, for every prime $p$, define
$$
D_p:=
\frac1{a_p}+\frac1{b_p}-\frac4{p+r_p}.
$$
It follows by the harmonic-arithmetic mean
inequality and the previous observation $a_p+b_p\le p+r_p$ that $D_p\ge 0$. 
Moreover, by \eqref{eq:rp-small}, we have
\begin{align*}
0\le \sum_{p\leq Q}
\left(
\frac1p-\frac1{p+r_p}
\right)\log p
\leq
\sum_{p\leq Q}\frac{r_p\log p}{p^2}
\ll
O(1)+\sum_{p\leq Q}\frac1p\ll
\log\log Q.
\end{align*}
Therefore,
using the classical estimate
$$
\sum_{p\leq Q}\frac{\log p}{p}
=
\log Q+O(1),
$$
we obtain
\begin{align}
\sum_{p\leq Q}\frac{\log p}{p+r_p}
&=
\sum_{p\leq Q}\frac{\log p}{p}
-
\sum_{p\leq Q}
\left(
\frac1p-\frac1{p+r_p}
\right)\log p
\notag\\
&\geq
\log Q-O(\log\log Q).
\label{eq:prp-lower}
\end{align}
Consequently, by \eqref{eq:larger-sieve-A} and
\eqref{eq:prp-lower},
\begin{equation}\label{eq:Dp-sum}
\begin{split}
\sum_{p\leq Q}D_p\log p
&=
\sum_{p\leq Q}
\left(
\frac1{a_p}+\frac1{b_p}
\right)\log p
-
4\sum_{p\leq Q}\frac{\log p}{p+r_p}
\\
&\leq
4\log Q+O(\log\log Q)
-
4\bigl(\log Q-O(\log\log Q)\bigr)
\\
&\ll
\log\log Q.
\end{split}
\end{equation}

We now relate $D_p$ to the deviations of $a_p$ and $b_p$ below
$p/2$. Suppose first that $a_p<p/2$, and put $s:=p+r_p$. Since $b_p\leq s-a_p$, we have
\begin{align*}
D_p
&\geq
\frac1{a_p}+\frac1{s-a_p}-\frac4s
=
\frac{(s-2a_p)^2}{s\,a_p(s-a_p)}.
\end{align*}
Recalling that
$d_{A,p}=(1/2-a_p/p)_+$, we have
$s-2a_p\geq p-2a_p=2p\,d_{A,p}$. Moreover, since
$r_p\leq p$, we have $s\leq2p$, while
$a_p(s-a_p)\leq s^2/4$. Hence
\[
s\,a_p(s-a_p)\leq\frac{s^3}{4}\leq2p^3,
\]
and consequently 
$
D_p\geq 2d_{A,p}^2/p.
$ 
If $a_p\geq p/2$, then $d_{A,p}=0$, so the same inequality holds
trivially. By symmetry, we have also 
$
D_p\geq 2d_{B,p}^2/p.
$ 
It follows that
\[
D_p\geq 
\max\left\{\frac{2d_{A,p}^2}{p}, \frac{2d_{B,p}^2}{p}\right\}\geq
\frac{d_{A,p}^2+d_{B,p}^2}{p}.
\]
Thus, we obtain by \eqref{eq:Dp-sum} that 
\begin{equation}\label{eq:L2-defect}
\sum_{p\leq Q}
\left(d_{A,p}^2+d_{B,p}^2\right)
\frac{\log p}{p}
\ll
\log\log Q.
\end{equation}

Notice that the quantities $d_{A,p}$ and $d_{B,p}$ depend only on
the fixed sets $A$ and $B$, and not on the auxiliary parameter $Q$.

For each $j\geq1$, define
\[
I_j:=\{p\in \mathbb{P}:e^{2^j}<p\leq e^{2^{j+1}}\}.
\]
Applying \eqref{eq:L2-defect} with $Q=e^{2^{j+1}}$, we obtain
\[
\sum_{p\in I_j}
d_{A,p}^2\frac{\log p}{p}
\ll j+1,
\]
and similarly for $B$. Moreover, partial summation together with
Mertens' estimate gives
\[
\sum_{p\in I_j}\frac1{p\log p}\ll \frac{1}{2^{j}}.
\]
Hence, by the Cauchy--Schwarz inequality,
\begin{align*}
\sum_{p\in I_j}\frac{d_{A,p}}p
\leq
\left(
\sum_{p\in I_j}
d_{A,p}^2\frac{\log p}{p}
\right)^{1/2}
\left(
\sum_{p\in I_j}
\frac1{p\log p}
\right)^{1/2}
\ll
\frac{\sqrt{j+1}}{2^{j/2}},
\end{align*}
and the same estimate holds with $A$ replaced by $B$.

Finally, let $J$ be such that
$e^{2^J}<Z\leq e^{2^{J+1}}$. Then
$2^J\asymp\log Z$ and $J\asymp\log\log Z$. Therefore
\begin{align*}
\sum_{p>Z}\frac{d_{A,p}+d_{B,p}}p
\ll
\sum_{j\geq J}\frac{\sqrt{j+1}}{2^{j/2}}
\ll
\frac{\sqrt{J+1}}{2^{J/2}}
\ll
\sqrt{\frac{\log\log Z}{\log Z}}.
\end{align*}
This concludes the proof. 
\end{proof}

\begin{lem}\label{lem:squarefree-mean-value}
Let $f$ be a multiplicative function supported on the squarefree
integers and satisfying 
$
0\leq f(p)\leq1
$ 
for every prime $p$. Suppose also that
$$
\sum_{p\in \mathbb{P}}\frac{1-f(p)}p<\infty.
$$
Then
$$%\begin{equation}\label{eq:squarefree-mean-value}
\lim_{x\to\infty}\frac{1}{x}\sum_{n\leq x}f(n)=
%\alpha_f:=
\prod_{p\in \mathbb{P}}
\left(1-\frac1p\right)
\left(1+\frac{f(p)}p\right)>0. 
$$%\end{equation}
\end{lem}

\begin{proof}
It follows as a consequence of \cite[Theorem 2]{MR169829}. 
\end{proof}

We can now prove the first main result.

\begin{proof}[Proof of Theorem \ref{thm:main2}]
Fix $A,B\subseteq \mathbb{N}$ with cardinality at least $2$, and suppose that
$A+B=_\star\mathbb P$. By \cite[Theorem 2.6]{MR3378834}, we have
\begin{equation}\label{eq:necforlemma}
A(x),B(x)
\gg
\frac{\sqrt{x}}{\log x\log\log x}
\gg
\frac{\sqrt{x}}{(\log x)^2}
\quad \text{ as }x\to \infty.
\end{equation}
Since $(A+B)\setminus\mathbb P$ is finite, there exist cofinite subsets
$A_0\subseteq A$ and $B_0\subseteq B$ such that
\[
A_0+B_0\subseteq\mathbb P.
\]
Then $A_0(x)=A(x)+O_{A,B}(1) \gg \sqrt{x}/(\log x)^2$ as $x\to \infty$, and similarly for $B$. In the following, we shall prove the required estimates for $A_0$ and $B_0$.
Since they are cofinite subsets of $A$ and $B$, respectively, the same
estimate will then hold for $A$ and $B$. Hence, for notational simplicity, we assume hereafter that $A=A_0$ and $B=B_0$.

For each prime $p$, define $A_p$, $a_p$, $d_{A,p}$ (and similarly for $B$) as in the proof of Lemma \ref{lem:tail-defect}. 
Note that $1\le a_p\le p-1$: in fact, if $a_p=p$, it is possible to choose $b \in B$ with $b>p$ and $a \in A$ such that $a\equiv -b\bmod{p}$. But this would be impossible because $p\mid a+b$, $a+b$ is prime, and $a+b\ge b>p$.

Define a multiplicative function $f$, supported on the squarefree
integers, such that 
\begin{equation}\label{eq:def-fA}
\forall p \in \mathbb{P}, \qquad 
f(p):=
\min
\left\{
1,\frac{a_p}{p-a_p}
\right\}.
\end{equation}
If 
$
a_p\geq p/2,
$ 
then 
$
f(p)=1.
$ 
If 
$
a_p<p/2,
$ 
then
\begin{align*}
1-f(p)
=
\frac{p-2a_p}{p-a_p}
=\frac{2\left(\frac{1}{2}-\frac{a_p}{p}\right)}{1-\frac{a_p}{p}}
\leq
4\left(
\frac12-\frac{a_p}{p}
\right)
=
4d_{A,p}.
\end{align*}
Hence, taking also into account \eqref{eq:necforlemma}, Lemma \ref{lem:tail-defect} gives
$$%\begin{equation}\label{eq:fA-summability}
\sum_{p \in \mathbb{P}}\frac{1-f(p)}p
\ll 
\sum_{p \in \mathbb{P}}\frac{d_{A,p}}p
<\infty.
$$%\end{equation}

Applying Lemma \ref{lem:squarefree-mean-value}, we obtain
\begin{equation}\label{eq:fA-positive-mean}
\sum_{n\leq y}f(n)
\sim
\alpha y 
\quad \text{ as }y\to \infty, 
\end{equation}
where
$$%\begin{equation}\label{eq:def-cA}
\alpha
:=
\prod_{p\in \mathbb{P}}
\left(1-\frac1p\right)
\left(1+\frac{f(p)}p\right)>0.
$$%\end{equation}
Notice that the value of $\alpha$ depends on the choice of the pair $(A,B)$.

We now use the Montgomery's large-sieve, see \cite[Lemma 3.2]{MR3378834} for the statement. 
To this end, define
$$
y:=\lfloor\sqrt{x}\rfloor. 
%\quad \text{ and }\quad 
%B^\flat:=B\cap(y,x].
$$
For every prime $p\leq y$, the set $B\cap(y,x]$ avoids all residue classes in $-A_p$ (for the same reasoning above). 
Thus, $|B\cap(y,x] \bmod{p}| \le 
p-a_p$. Therefore the large-sieve estimate gives
\begin{align*}
|B\cap(y,x]|
\ll
\frac{x+y^2}
{\displaystyle
\sum_{\substack{q\leq y\\q\ {\rm squarefree}}}
\prod_{p\mid q}
\frac{a_p}{p-a_p}}
\leq
\frac{x+y^2}
{\displaystyle
\sum_{q\leq y}f(q)}.
\end{align*}
By \eqref{eq:fA-positive-mean}, we have 
$
\sum_{q\leq y}f(q)
\gg_{A,B}
y,
$ 
and hence
\begin{equation}\label{eq:estimB}
B(x)\le 
y+|B\cap(y,x]|\ll_{A,B}
y+\frac{x+y^2}{y}
\ll_{A,B}
\sqrt{x}
\quad \text{ as }x\to \infty.
\end{equation}
Interchanging $A$ and $B$, we obtain also $A(x)
\ll_{A,B}
\sqrt{x}$ as $x\to \infty$. 

To complete the proof, observe that every sufficiently large prime $p\le x$ can be represented as $a+b$ for some $a \in A$ and $b \in B$ with $a,b\le x$. It follows from the upper bound \eqref{eq:estimB} (and the classical estimates on $\pi(x):=\mathbb{P}(x)$) that 
$$
\frac{x}{\log x}\ll \pi(x)-O_{A,B}(1) \le A(x)B(x) \ll_{A,B} \, A(x)\sqrt{x}
\quad 
\text{ as }x\to \infty. 
$$
Therefore $A(x) \gg_{A,B} \sqrt{x}/\log x$ as $x\to \infty$, and analogously for $B$. 
\end{proof}

\begin{comment}
\begin{rmk}\label{rmk:comparison-EH-main2}
The proof of Theorem \ref{thm:main2} uses the same two sieve-theoretic
ingredients as the proof of \cite[Theorem 2.6]{MR3378834}: the
larger sieve is used to control the residue supports modulo primes,
while the final upper bound is obtained from the large sieve.
For completeness, the required larger-sieve estimate was derived
directly in \eqref{eq:elementary-larger-sieve}; the final large-sieve
step was quoted in the precise form needed from \cite[Lemma
3.2]{MR3378834}.

The difference is in the use of the residue information between these
two sieve steps. In \cite[Section 7]{MR3378834}, the argument is
performed at one scale and yields a quadratic estimate for the
deviations of the residue supports from $p/2$. This is then combined
with Hildebrand's theorem at the same scale, producing the factor
$\log\log x$ in \eqref{eq:startingbound}.

Here the sets $A$ and $B$ are fixed infinite sets. Hence the residue
supports
$$
A\bmod p
\qquad\text{and}\qquad
B\bmod p
$$
are fixed while the auxiliary parameter $Q$ tends to infinity.
This allows us to apply the same larger-sieve estimate at
arbitrarily many scales and to deduce the summability
$$
\sum_p
\frac{d_{A,p}+d_{B,p}}p
<\infty.
$$
The resulting multiplicative function $f_A$ therefore has a positive
mean value
$$
\sum_{n\leq X}f_A(n)
\sim
\mathfrak c_A X,
$$
rather than a lower bound losing a factor $\log\log X$. This removes
the latter factor from the final large-sieve estimate and gives
$$
A(x),B(x)\ll_{A,B}\sqrt{x}.
$$
\end{rmk}
\end{comment}

\medskip

\begin{proof}
[Proof of Corollary \ref{cor:kappa}]
Thanks to Theorem \ref{thm:main2}, there exist $c_1,c_2>0$, depending
only on $A$ and $B$, such that
$$
c_1\frac{\sqrt{x}}{\log x}\leq A(x),B(x)\leq c_2\sqrt{x} 
\quad \text{ as }x\to \infty.
$$
Since $\kappa(x)\to\infty$, we have
$\kappa(x)\geq\max\{1/c_1,c_2\}$ for all sufficiently large $x$, and
the assertion follows.
\end{proof}

\section{Proof of Theorem \ref{thm:main3} and Corollary \ref{cor:powers}}\label{sec:proof-main3}

We begin with two auxiliary sieve estimates. 

\begin{lem}\label{lem:many-shifts-main3}
There exist absolute constants $c_0,C_0>0$ and an integer $h_0\geq2$ such that, for all
sufficiently large $x$, the following assertion holds. Let
$$
h_0\leq h\leq c_0\,x^{1/5}
$$
and let $a_1,\ldots,a_h$ be distinct integers in $[0,x/4]$. Suppose $V\subseteq[2x^{1/4},x/4]\cap\mathbb N$ and $M\geq0$
satisfy
$$
|\{n\in V:n+a_i\notin\mathbb P\}|\leq M
$$
for all $i \in \{1,\ldots,h\}$. Then
\begin{equation}\label{eq:many-shifts-main3}
|V|\leq C_0\left(\frac{x}{h^2}+M\right).
\end{equation}
\end{lem}

\begin{proof}
Set $z:=x^{1/5}$, $y:=x^{1/4}$, and $\mathcal P:=\{p\in\mathbb P:z<p\leq y\}$. 
For each $p\in\mathcal P$, let $\Omega_p$ be the set of residue
classes modulo $p$ represented by $-a_1,\ldots,-a_h$, and write
$$
%\omega_p:=|\Omega_p|,
%\qquad
w_p:=\frac{|\Omega_p|}{p}
\qquad\text{and}\qquad
w:=\sum_{p\in\mathcal P}w_p.
$$
For an integer $n$, membership in $\Omega_p$ means that $p$ divides
at least one of the sums $n+a_1,\ldots,n+a_h$. We shall first bound
the number of integers for which this happens for few primes
$p\in\mathcal P$.

By Mertens' estimate, we have
$$
\sum_{p\in\mathcal P}\frac1p=\log(\nicefrac{5}{4})+o(1)\qquad \text{ as }x\to \infty.
$$
To estimate $w$, observe that $|\Omega_p|$ can be smaller than $h$ only because some of the shifts coincide modulo $p$. More precisely,
$$
h-|\Omega_p|
\leq
\sum_{1\leq i<j\leq h}\mathbf1_{\{p\mid a_i-a_j\}}.
$$
Each difference $a_i-a_j$ is nonzero and has absolute value at most
$x$. It therefore has at most five distinct prime divisors exceeding
$z=x^{1/5}$. Since $\nicefrac{1}{p}\leq\nicefrac{1}{z}$ for $p\in\mathcal P$, it follows that
$$
\begin{aligned}
w
&=h\sum_{p\in\mathcal P}\frac1p
  -\sum_{p\in\mathcal P}\frac{h-|\Omega_p|}{p}\geq
h\sum_{p\in\mathcal P}\frac1p-\frac5z\binom h2.
\end{aligned}
$$
Together with the upper bound
$w\leq h\sum_{p\in\mathcal P}\nicefrac{1}{p}$, this shows that, choosing $c_0>0$ sufficiently small, we have
\begin{equation}\label{eq:lambda-many-shifts}
w\asymp h
\end{equation}
uniformly for $h\leq c_0x^{1/5}$. We may then choose  $h_0$ sufficiently large that $w\geq8$ whenever
$h\geq h_0$.

For each integer $n$, write
$$
R(n):=\sum_{p\in\mathcal P}\mathbf1_{\Omega_p}(n\bmod p).
$$
Thus $R(n)$ counts the primes in $\mathcal P$ dividing at least one
of the sums $n+a_i$. The average of
$\mathbf1_{\Omega_p}(n\bmod p)$ over a complete residue system
modulo $p$ is $w_p$. Accordingly, define 
$$
\psi_p(n):=w_p-\mathbf1_{\Omega_p}(n\bmod p)
\quad\text{ and }\quad 
T(n):=\sum_{\substack{p,q\in\mathcal P\\p<q}}
\psi_p(n)\psi_q(n).
$$
We use $T$ to detect the integers for which $R(n)$ is much smaller
than $w$. Indeed, if $R(n)\leq w/4$, then
$$
\sum_{p\in\mathcal P}\psi_p(n)
=w-R(n)\geq\frac{3w}{4}.
$$
Moreover, since $0\leq w_p\leq1$, we have
$$
\psi_p(n)^2
\leq w_p^2+\mathbf1_{\Omega_p}(n\bmod p),
$$
and hence
$$
\sum_{p\in\mathcal P}\psi_p(n)^2
\leq\sum_{p\in\mathcal P}w_p^2+R(n)
\leq w+\frac{w}{4}
=\frac{5w}{4}.
$$
Expanding the square of $\sum_p\psi_p(n)$ and using $w\geq8$,
we therefore obtain
\begin{equation}\label{eq:T-large-many-shifts}
T(n)
=\frac12\left(
\left(\sum_{p\in\mathcal P}\psi_p(n)\right)^2
-\sum_{p\in\mathcal P}\psi_p(n)^2
\right)\geq\frac{9w^2}{32}-\frac{5w}{8}
\geq\frac{w^2}{8}.
\end{equation}

At this point, write $e(t):=\exp(2\pi it)$. Since $\psi_p$ is periodic modulo $p$
and has mean zero, its finite Fourier expansion has the form
$$
\psi_p(n)=\sum_{a=1}^{p-1}c_{p,a}e\left(\frac{an}{p}\right).
$$
By Parseval's identity,
$$
\sum_{a=1}^{p-1}|c_{p,a}|^2
=\frac1p\sum_{r=0}^{p-1}|\psi_p(r)|^2
=w_p(1-w_p).
$$
Notice that, for all distinct primes $p,q\in\mathcal P$, we have
$$
\psi_p(n)\psi_q(n)
=\sum_{a=1}^{p-1}\sum_{b=1}^{q-1}
c_{p,a}c_{q,b}
e\left(n\frac{aq+bp}{pq}\right).
$$
Note that the fraction $(aq+bp)/(pq)$ 
is already in lowest terms. In addition, 
for a fixed pair $p,q$, different choices of $(a,b)$ give
different frequencies modulo $1$, as is seen by reducing the
numerator modulo $p$ and modulo $q$. Different unordered pairs
$\{p,q\}$ also give different frequencies, since their reduced
denominators are different. 
Hence, the sum of the squared absolute values of all the
Fourier coefficients of $T$ is
$$
\begin{aligned}
\sum_{\substack{p,q\in\mathcal P\\p<q}}
\sum_{a=1}^{p-1}\sum_{b=1}^{q-1}
|c_{p,a}c_{q,b}|^2
&=\sum_{\substack{p,q\in\mathcal P\\p<q}}
w_p(1-w_p)w_q(1-w_q)
\\&
\leq\frac12\left(\sum_{p\in\mathcal P}w_p\right)^2
=\frac{w^2}{2}.
\end{aligned}
$$
Put $N:=\lfloor x\rfloor+1$ and write the Fourier expansion of $T$ as
$$
T(n)=\sum_{j=1}^J d_j e(n\theta_j),
$$
where the $\theta_j\in[0,1)$ are distinct reduced fractions with denominators
at most $\sqrt{x}\leq\sqrt{N}$. The preceding calculation gives
$$
\sum_{j=1}^J|d_j|^2\leq\frac{w^2}{2}.
$$
Define 
$
S:=\sum_{n=0}^{N-1}|T(n)|^2.
$ 
Substituting the Fourier expansion for one factor of $T(n)$
and applying the Cauchy--Schwarz inequality, we obtain
$$
\begin{aligned}
S^2
&=\left|
\sum_{j=1}^J d_j
\sum_{n=0}^{N-1}\overline{T(n)}e(n\theta_j)
\right|^2\\
&\leq
\left(\sum_{j=1}^J|d_j|^2\right)
\sum_{j=1}^J
\left|\sum_{n=0}^{N-1}\overline{T(n)}e(n\theta_j)\right|^2.
\end{aligned}
$$
%%%%%%%%%%%%%%%%%%%%%%%%%%%%%%%%%
At this point, for each $j \in \{1,\ldots,J\}$ choose coprime integers $1\le r_j<q_j$ such that $\theta_j=\frac{r_j}{q_j}$. 
%$$
%\theta_j=\frac{r_j}{q_j},
%\qquad
%1\leq r_j<q_j,
%\qquad
%(r_j,q_j)=1
%\quad\text{for all }j\in\{1,\ldots,J\}.
%$$
Notice that $q_j\leq\sqrt{x}\leq\sqrt{N}$ and that the pairs
$(q_j,r_j)$ are distinct. 
Define also
$$
u_m:=\overline{T(m-1)}
\quad\text{for all }m\in\{1,\ldots,N\}.
$$
Applying the large sieve inequality (see \cite{MR260703} or \cite[Lemma 2]{Els26}) with these coefficients, we get 
$$
\sum_{d\leq\sqrt{N}}
\sum_{\substack{1\leq a\leq d\\(a,d)=1}}
\left|\sum_{m=1}^{N}u_m e(ma/d)\right|^2
\leq 2N\sum_{m=1}^{N}|u_m|^2
=2NS.
$$
Observe that each pair $(q_j,r_j)$
occurs among the pairs $(d,a)$ in the left-hand side.
Since these pairs are distinct and all the summands are
nonnegative, it follows that
$$
\sum_{j=1}^{J}
\left|\sum_{m=1}^{N}u_m e(m\theta_j)\right|^2
\leq 2NS.
$$
On the other hand, changing variables $m=n+1$, we obtain
$$
\begin{aligned}
\sum_{n=0}^{N-1}\overline{T(n)}e(n\theta_j)
&=\sum_{m=1}^{N}u_m e((m-1)\theta_j)
%\\&
=e(-\theta_j)\sum_{m=1}^{N}u_m e(m\theta_j).
\end{aligned}
$$
Since $|e(-\theta_j)|=1$, the last two estimates imply that
$$
\sum_{j=1}^{J}
\left|\sum_{n=0}^{N-1}\overline{T(n)}e(n\theta_j)\right|^2
\leq 2NS.
$$
Putting together this estimate and the preceding
Cauchy--Schwarz inequality, we get
$$
S^2
\leq
\left(\sum_{j=1}^{J}|d_j|^2\right)2NS
\leq Nw^2S.
$$
If $S>0$, this implies that $S\leq Nw^2$, and the same
inequality holds if $S=0$. Hence
\begin{equation}\label{eq:T-moment-many-shifts}
\sum_{0\leq n\leq x}|T(n)|^2
=S\leq Nw^2\ll xw^2.
\end{equation}
\begin{comment}
We now apply \cite[Lemma 2]{Els26} with
coefficients $u_m:=\overline{T(m-1)}$ for $1\leq m\leq N$, cf. also \cite{MR260703}.
Indeed,
$$
\sum_{n=0}^{N-1}\overline{T(n)}e(n\theta_j)
=e(-\theta_j)\sum_{m=1}^{N}u_m e(m\theta_j),
$$
and $|e(-\theta_j)|=1$. Since the $\theta_j$ form a subset of
the reduced fractions with denominators at most $\sqrt{N}$,
that lemma gives
$$
\sum_{j=1}^J
\left|\sum_{n=0}^{N-1}\overline{T(n)}e(n\theta_j)\right|^2
\leq 2N\sum_{m=1}^{N}|u_m|^2
=2NS.
$$
Hence, $S^2\leq Nw^2S$. Dividing by $S$ when $S>0$
(the case $S=0$ being immediate), we conclude that
\begin{equation}\label{eq:T-moment-many-shifts}
\sum_{0\leq n\leq x}|T(n)|^2
=S\leq Nw^2\ll xw^2.
\end{equation}
\end{comment}
By \eqref{eq:T-large-many-shifts}, each integer $n$ satisfying
$R(n)\leq w/4$ contributes at least $w^4/64$ to this sum.
Therefore 
$$
\begin{aligned}
|\{n\in[0,x]\cap\mathbb N:R(n)\leq w/4\}|
%&
\leq\frac{64}{w^4}\sum_{0\leq n\leq x}|T(n)|^2
%\\&
\ll\frac{x}{w^2}
\ll\frac{x}{h^2},
\end{aligned}
$$
where the last estimate follows from \eqref{eq:lambda-many-shifts}.

We miss only to bound the number of elements of $V$ for which
$R(n)>w/4$. For each $n\in V$, define 
$$
b(n):=|\{i\in\{1,\ldots,h\}:n+a_i\notin\mathbb P\}|.
$$
Summing the hypothesis over $i\in\{1,\ldots,h\}$, it follows that 
$$
\sum_{n\in V}b(n)
=\sum_{i=1}^h|\{n\in V:n+a_i\notin\mathbb P\}|
\leq hM.
$$
If $p\in\mathcal P$ is counted by $R(n)$, then $p\mid n+a_i$
for some $i$. Since $n+a_i\geq2y>p$, this sum is composite and
is therefore counted by $b(n)$. Each such sum is at most $x$,
so it has at most five distinct prime divisors in $\mathcal P$.
Thus, for all $n \in V$, we have 
$$
R(n)\leq5b(n).
$$
In particular, $R(n)>w/4$ implies $b(n)>w/20$.
It follows that
$$
|\{n\in V:R(n)>w/4\}|
\leq\frac{20}{w}\sum_{n\in V}b(n)
\leq\frac{20hM}{w}
\ll M.
$$
Every element of $V$ belongs to one of the two sets just estimated.
Combining their bounds, we obtain
$$
|V|\ll\frac{x}{h^2}+M,
$$
which concludes the proof.
\end{proof}

We shall also use the following well-known application of Selberg’s sieve upper bounds on the number of $k$-tuples of primes, for fixed $k$, as in \cite[Lemma 2.3]{MR4458542}. 
\begin{lem}\label{lem:fixed-shifts-main3}
Fix distinct $a_1,\ldots,a_k \in \mathbb{N}$ with $k\ge 2$. 
Then
$$
\left|
\left\{
n\in [0,x]:\ 
n+a_i\in\mathbb P
\text{ for each }i\in \{1,\ldots, k\}
\right\}
\right|
\ll_{a_1,\ldots,a_k}
\frac{x}{(\log x)^k}
$$
as $x\to\infty$.
\end{lem}
\begin{proof}
It follows by \cite[Theorem 7.16]{MR2647984}. 
\end{proof}

\begin{lem}\label{lem:robust-binary-main3}
Let $X,Y\subseteq\mathbb N$ have at least two elements each, and define 
$$
m(x):=\bigl((X+Y)\bigtriangleup\mathbb P\bigr)(x).
$$
Then, for all sufficiently large $x$ such that
\begin{equation}\label{eq:small-m-robust}
m(x)\leq\frac{x}{(\log x)^8},
\end{equation}
we have
$$
X(x/4),Y(x/4)\ll x^{3/5}+m(x)\quad \text{ as }x\to \infty.
$$
\end{lem}

\begin{proof}
Fix a sufficiently large $x$ satisfying \eqref{eq:small-m-robust},
and set
$$
m:=m(x),\qquad
u:=X(x/4),\quad \text{ and }\quad 
v:=Y(x/4).
$$
All but at most $m$ primes $p\leq x/4$ belong to $X+Y$, which implies that
\begin{equation}\label{eq:UV-product-main3}
uv\geq\pi(x/4)-m\gg\frac{x}{\log x}.
\end{equation}

First, we may suppose that both $X$ and $Y$ are infinite.
In fact, suppose that
$
X=\{a_1,\ldots,a_r\}
$
is finite. If $x$ is sufficiently large, each $a_i$ belongs to
$[0,x/4]$. For each $i$, at most $m$ elements
$y\in Y\cap[0,x/4]$ satisfy $y+a_i\notin\mathbb P$:
in fact, the map $y\mapsto y+a_i$ is injective and takes these elements
into $((X+Y)\setminus\mathbb P)\cap[0,x]$.
Thus, apart from at most $rm$ elements $y\in Y\cap[0,x/4]$,
all the integers
$
y+a_1,\ldots,y+a_r
$
are prime. Hence Lemma \ref{lem:fixed-shifts-main3} gives
$$
v\ll_X\frac{x}{(\log x)^r}+rm
\ll_X\frac{x}{(\log x)^2}.
$$
This is impossible if $x$ is sufficiently large, because
\eqref{eq:UV-product-main3} and $u=r$ imply that
$v\gg_X x/\log x$. Therefore, if $X$ is finite, there are no
sufficiently large $x$ satisfying \eqref{eq:small-m-robust}.
The same argument applies to $Y$.

Next, we claim that
\begin{equation}\label{eq:UV-log4}
u,v>(\log x)^4
\end{equation}
for all sufficiently large $x$ satisfying \eqref{eq:small-m-robust}.
Choose six distinct integers $a_1,\ldots,a_6\in X$.
For all sufficiently large $x$, they belong to $[0,x/4]$.
If $u\leq(\log x)^4$, then \eqref{eq:UV-product-main3} implies
$$
v\gg\frac{x}{(\log x)^5}.
$$
On the other hand, apart from at most $6m$ elements
$y\in Y\cap[0,x/4]$, all the integers
$
y+a_1,\ldots,y+a_6
$
are prime. Therefore, by Lemma \ref{lem:fixed-shifts-main3},
$$
v\ll_X\frac{x}{(\log x)^6}+6m
\ll_X\frac{x}{(\log x)^6},
$$
which is impossible if $x$ is sufficiently large.
Thus $u>(\log x)^4$, and symmetrically $v>(\log x)^4$.

At this point, define
$$
H:=\lfloor c_0x^{1/5}\rfloor
\qquad\text{and}\qquad
h:=\min\{u,H\},
$$
where $c_0$ is the constant appearing in
Lemma \ref{lem:many-shifts-main3}.
By \eqref{eq:UV-log4}, if $x$ is sufficiently large, then
$h\geq h_0$, where $h_0$ is the integer appearing in that lemma.
Choose $h$ distinct elements of $X\cap[0,x/4]$ as shifts and define 
$$
V:=Y\cap[2x^{1/4},x/4].
$$
For each chosen shift $a$, at most $m$ elements $n\in V$
satisfy $n+a\notin\mathbb P$. 
Applying Lemma \ref{lem:many-shifts-main3} with $M=m$, and
including the elements of $Y$ below $2x^{1/4}$, we obtain
\begin{equation}\label{eq:V-h-final}
v\ll x^{1/4}+\frac{x}{h^2}+m.
\end{equation}

If $u\geq H$, then $h=H\asymp x^{1/5}$, and consequently
$$
v\ll x^{1/4}+x^{3/5}+m
\ll x^{3/5}+m.
$$
Suppose now that $u<H$, so that $h=u$. In this case,
\eqref{eq:UV-product-main3} implies 
$$
v\gg\frac{x}{u\log x}.
$$
Let us compare the first two terms in \eqref{eq:V-h-final} with
the right-hand side of this lower bound. First, since
$u<H\leq c_0x^{1/5}$, we have
$$
\frac{x^{1/4}}{x/(u\log x)}
=\frac{u\log x}{x^{3/4}}
\leq\frac{c_0\log x}{x^{11/20}}
=o(1).
$$
Second, by \eqref{eq:UV-log4}, we get
$$
\frac{x/u^2}{x/(u\log x)}
=\frac{\log x}{u}
\leq\frac{1}{(\log x)^3}
=o(1).
$$
It follows that, for all sufficiently large $x$, the contributions
of these two terms to \eqref{eq:V-h-final} are together at most
$v/2$. We obtain that 
$
v\leq\frac{v}{2}+Cm
$ 
for some absolute constant $C>0$, and hence $v\ll m$. 
Thus, in both cases,
$$
v\ll x^{3/5}+m.
$$
The analogous argument gives the same estimate for $u$.
\end{proof}

We shall finally use the following submultiplicative inequality for
sumsets, due to Gyarmati, Matolcsi, and Ruzsa. 
%\cite[Theorem 1.2]{MR2676833}.

\begin{lem}\label{lem:sumset-submultiplicativity}
Let $r\geq2$ and let $E_1,\ldots,E_r$ be finite nonempty subsets of
$\mathbb N$. For each $i \in \{1,\ldots,r\}$, define 
$$
E_{-i}
:=
E_1+\cdots+E_{i-1}+E_{i+1}+\cdots+E_r.
$$
Then
\begin{equation}\label{eq:sumset-submultiplicativity}
|E_1+\cdots+E_r|^{r-1}
\leq
\prod_{i=1}^r |E_{-i}|.
\end{equation}
\end{lem}

\begin{proof}
This follows by \cite[Theorem 1.2]{MR2676833}.
\end{proof}

We can now prove the second main result.

\begin{proof}[Proof of Theorem \ref{thm:main3}]
Fix $k\geq3$ and sets $A_1,\ldots,A_k\subseteq\mathbb N$ with
cardinality at least $2$.
Fix $x$ sufficiently large and write
$$
m:=\left((A_1+\cdots+A_k)\bigtriangleup\mathbb P\right)(x).
$$
If
$
m>x/(\log x)^8,
$
then $m\gg x^c$ for each $c\in(0,1)$, hence
\eqref{eq:finalclaim} holds. Thus, we can assume hereafter that
$m\leq x/(\log x)^8$.

For each $i\in\{1,\ldots,k\}$, define
$$
A_{-i}:=
A_1+\cdots+A_{i-1}+A_{i+1}+\cdots+A_k,
$$
so that $A_1+\cdots+A_k=A_{-i}+A_i$.
Since both $A_{-i}$ and $A_i$ have cardinality at least $2$, we may
apply Lemma \ref{lem:robust-binary-main3} to the binary decomposition
$A_{-i}+A_i$. Notice that the corresponding error term is precisely
$m$, which is independent of $i$. Thus, setting
$$
Q(x,m):=x^{3/5}+m,
$$
we obtain
\begin{equation}\label{eq:all-subsums-upper-main3}
A_{-1}(x/4),\ldots,A_{-k}(x/4)\ll Q(x,m).
\end{equation}

At this point, define
$$
z:=
\left\lfloor
\frac{x}{4(k-1)}
\right\rfloor
\quad\text{and}\quad
E_i:=A_i\cap[0,z]\text{ for all }i\in\{1,\ldots,k\}.
$$
Observe that all but at most $m$ primes $p\leq z$ belong to
$E_1+\cdots+E_k$, since any representation of such a prime as
a sum of elements of the $A_i$ uses only summands in $[0,z]$.
Hence
\begin{equation}\label{eq:kfold-lower-main3}
|E_1+\cdots+E_k|
\geq
\pi(z)-m
\gg_k
\frac{x}{\log x}.
\end{equation}

Define also the finite sets
$E_{-i}:=E_1+\cdots+E_{i-1}+E_{i+1}+\cdots+E_k$
for all $i\in\{1,\ldots,k\}$. Using that
$\max E_{-i}\leq\sum_{j\neq i}\max E_j$, we obtain
$$
E_{-i}\subseteq A_{-i}\cap[0,(k-1)z]
\subseteq A_{-i}\cap\left[0,\frac{x}{4}\right]
$$
for each $i\in\{1,\ldots,k\}$, so that
$|E_{-i}|\ll Q(x,m)$ by \eqref{eq:all-subsums-upper-main3}.
It follows by \eqref{eq:kfold-lower-main3} and
Lemma \ref{lem:sumset-submultiplicativity} that
\begin{equation}\label{eq:Qk-upper-main3333}
\left(\frac{x}{\log x}\right)^{k-1}
\ll_k |E_1+\cdots+E_k|^{k-1}
\leq\prod_{i=1}^k|E_{-i}|
\ll_k Q(x,m)^k.
\end{equation}
Now, since $(a+b)^k\leq2^{k-1}(a^k+b^k)$ for all $a,b\geq0$
(and $k$ is fixed), we have that
\begin{equation}\label{eq:Qk-upper-main3}
Q(x,m)^k\ll_k x^{3k/5}+m^k.
\end{equation}
Recalling that $k\geq3$, it follows that $3k/5<k-1$ and hence
$x^{3k/5}=o((x/\log x)^{k-1})$.

Putting together \eqref{eq:Qk-upper-main3333} and
\eqref{eq:Qk-upper-main3} and the last observation, we get
$$
\left(\frac{x}{\log x}\right)^{k-1}\ll_k m^k.
$$
This implies that
\begin{equation}\label{eq:m-lower-main3}
m\gg_k\left(\frac{x}{\log x}\right)^{1-1/k},
\end{equation}
which is exactly our claim \eqref{eq:finalclaim}.
\end{proof}

\medskip

\begin{proof}
    [Proof of Corollary \ref{cor:powers}] 
    Note that 
    $$
    (P^\sharp \bigtriangleup \mathbb{P})(x) \ll \frac{\sqrt{x}}{\log x} 
    \quad \text{ as }x\to \infty.
    $$
    Setting $S:=A_1+\cdots+A_k$, it follows by the triangle inequality and Theorem \ref{thm:main3} that 
    \begin{displaymath}
        \begin{split}
            (S\bigtriangleup P^\sharp)(x)&\ge (S\bigtriangleup \mathbb{P})(x)-(\mathbb{P} \bigtriangleup P^\sharp)(x)\\
            &\gg \left(\frac{x}{\log x}\right)^{1-1/k}-\frac{\sqrt{x}}{\log x}
            \gg \left(\frac{x}{\log x}\right)^{1-1/k}  
            \quad \text{ as }x\to \infty.
        \end{split}
    \end{displaymath}
This completes the proof. 
\end{proof}

%\subsection{Declarations} I declare that no AI has been used in the preparation of this manuscript. %, except for the final grammar check. 

\bibliographystyle{amsplain}
\bibliography{refs}

\end{document}